\documentclass[preprint,nonatbib]{elsarticle}
\usepackage{fullpage}
\usepackage{standalone}
\usepackage{import}

\makeatletter
\let\c@author\relax
\makeatother

\usepackage[table]{xcolor}
\usepackage{amsmath}
\usepackage{amsthm}
\usepackage{amssymb} 
\usepackage{hyperref}
\usepackage{graphicx}
\usepackage{float}
\usepackage{tikz-cd}
\usepackage{caption}
\usepackage{subcaption}
\usepackage{array}
\usepackage{extarrows}
\usepackage{multirow}
\usepackage{tikz}
\usepackage{makecell}
\usepackage{hhline}
\usepackage{ mathdots }
\usepackage{cleveref}
\usepackage{csquotes}
\usepackage{import}
\usepackage[normalem]{ulem}
\usepackage{enumitem}
\usepackage{setspace}
\usetikzlibrary{cd}
\usetikzlibrary{graphs}
\usetikzlibrary {graphs.standard}
\usetikzlibrary{shapes.geometric}

\DeclareMathOperator{\LS}{RP}

\DeclareMathOperator{\ILS}{ILS}

\DeclareMathOperator{\ROS}{ROS}

\allowdisplaybreaks[3]

\newtheorem{theorem}{Theorem}[section]
\newtheorem{lemma}[theorem]{Lemma}

\theoremstyle{definition}
\newtheorem{definition}[theorem]{Definition}

\newtheoremstyle{noparentheses}
  {\topsep}   
  {\topsep}   
  {\itshape}  
  {0pt}       
  {\bfseries} 
  {\textbf{.}}         
  {5pt plus 1pt minus 1pt} 
  {\thmname{#1} \thmnumber{#2} \thmnote{\normalfont#3}}          
\theoremstyle{noparentheses}
\newtheorem{theorem*}[theorem]{Theorem}
\newtheorem{lemma*}[theorem]{Lemma}
\newtheorem{conjecture*}[theorem]{Conjecture}

\Crefname{conjecture}{Conjecture}{Conjectures}
\Crefname{conjecture*}{Conjecture}{Conjectures}
\Crefname{append}{Appendix}{Appendices}
\Crefname{theorem*}{Theorem}{Theorems}

\numberwithin{equation}{section}

\usepackage{biblatex} 
\hypersetup{pdfauthor=author}

\makeatletter
\def\ps@pprintTitle{%
  \let\@oddhead\@empty
  \let\@evenhead\@empty
  \def\@oddfoot{\reset@font\hfil\thepage\hfil}
  \let\@evenfoot\@oddfoot
}
\makeatother

\begin{document}

\title{Latin squares with non-partitioning disjoint subsquares}
\author{Tara Kemp}
\ead{t.kemp@uq.net.au}

\address{School of Mathematics and Physics, ARC Centre of Excellence, Plant Success in Nature and Agriculture, The University of Queensland, Brisbane, Queensland, 4072, Australia}

\begin{abstract}
A latin square of order $n$ with pairwise disjoint subsquares of orders $h_1,\dots,h_k$ such that $h_1+\dots+h_k = n$ is known as a realization. The existence of realizations is a partially solved problem with a few general results for an arbitrary number of subsquares, $k$. Requiring only that $h_1+\dots+h_k\leq n$ gives a variation of the problem that has few known results. In this paper we prove a general necessary condition for existence and completely determine existence when there are at most three subsquares or the subsquares are all of the same order. Importantly, we prove that if $h_1\geq h_2\geq\dots\geq h_k$ and $n\geq h_1+\sum_{i=1}^kh_i$ then such a latin square always exists.
\end{abstract}


\maketitle

\section{Preliminaries}

A \emph{latin square} of order $n$ is an $n\times n$ array $L$ filled with symbols from $[n]=\{1,2,\dots,n\}$ such that each symbol occurs exactly once in every row and column. A \emph{subsquare} is an $m\times m$ subarray of $L$ which is itself a latin square of order $m$ on some set of $m$ symbols. Subsquares are disjoint if they share no rows, columns or symbols.

The latin square in \Cref{fig: ILS example} is a latin square of order 8 with pairwise disjoint subsquares of orders 3, 2 and 1 highlighted.

\begin{figure}[h]
    \centering
    $\arraycolsep=4pt\begin{array}{|c|c|c|c|c|c|c|c|}\hline
    \cellcolor{lightgray}1 & \cellcolor{lightgray}2 & \cellcolor{lightgray}3 & 6 & 7 & 8 & 5 & 4\\ \hline
    \cellcolor{lightgray}2 & \cellcolor{lightgray}3 & \cellcolor{lightgray}1 & 8 & 6 & 5 & 4 & 7\\ \hline
    \cellcolor{lightgray}3 & \cellcolor{lightgray}1 & \cellcolor{lightgray}2 & 7 & 8 & 4 & 6 & 5\\ \hline
    8 & 7 & 6 & \cellcolor{lightgray}4 & \cellcolor{lightgray}5 & 1 & 2 & 3\\ \hline
    7 & 6 & 8 & \cellcolor{lightgray}5 & \cellcolor{lightgray}4 & 2 & 3 & 1\\ \hline
    5 & 8 & 4 & 3 & 1 & \cellcolor{lightgray}6 & 7 & 2\\ \hline
    6 & 4 & 5 & 2 & 3 & 7 & 1 & 8\\ \hline
    4 & 5 & 7 & 1 & 2 & 3 & 8 & 6\\ \hline
    \end{array}$
    \caption{A latin square of order 8 with disjoint subsquares}
    \label{fig: ILS example}
\end{figure}

Given an integer partition $P=(h_1\dots h_k)$ of $n$, a \emph{realization} of $P$, denoted $\LS(h_1\dots h_k)$, is a latin square of order $n$ with pairwise disjoint subsquares of orders $h_1,\dots,h_k$.

The study of realizations began with a question posed by L.~Fuchs \cite{keedwell2015latin} regarding quasigroups with disjoint subquasigroups, and the problem of existence of a realization is partially solved. Realizations are known by a few names including \emph{partitioned incomplete latin squares} (PILS). This name corresponds to a variation of the problem which considers \emph{incomplete latin squares} (ILS): latin squares with pairwise disjoint subsquares that do not necessarily partition the order of the latin square. It is this variation that we consider here.

Given an integer partition $(h_1\dots h_k)$ of $m$, an \emph{incomplete latin square}, denoted by $\ILS(n;h_1\dots h_k)$, is a latin square of order $n\geq m$ with $k$ pairwise disjoint subsquares of orders $h_1,\dots,h_k$. The latin square in \Cref{fig: ILS example} is an $\ILS(8;3^12^11^1)$.

Note that the term \emph{disjoint} is sometimes used to refer to subsquares that do not intersect (do not share any cells but may share symbols, rows or columns). Our definition of disjoint is inherited from quasigroups and should not be confused with non-intersecting. The author would also like to acknowledge that the name \emph{incomplete latin square} is unsatisfactory since it is sometimes used to refer to partial latin squares and it does not effectively describe the object it refers to.

Unless otherwise stated, we assume that $h_1\geq h_2\geq \dots\geq h_k > 0$. Also, the partition notation $(h_1^{m_1}h_2^{m_2}\dots h_\ell^{m\ell})$ represents a partition with $m_i$ parts of size $h_i$ for all $i\in[\ell]$. Throughout, we assume that all realizations and incomplete latin squares are in \emph{normal form}: where the subsquares appear along the main diagonal, the $i^{th}$ subsquare is of order $h_i$, and for $i<j$ the symbols from the $i^{th}$ subsquare are less than the symbols from the $j^{th}$ subsquare.

There isn't much known about incomplete latin squares where the subsquares do not partition the order of the latin square. For ILS with a single subsquare, i.e. $k=1$, existence was determined by Evans \cite{evans1960embedding}.

\begin{theorem*}[\cite{evans1960embedding}]
\label{thm: single subsquare}
    There exists a latin square of order $n$ with a subsquare of order $m<n$ if and only if $m\leq \frac{n}{2}$.
\end{theorem*}

All other work on incomplete latin squares comes from research into mutually orthogonal latin squares with holes. For example, \cite{dukes2015mutually} considers mutually orthogonal latin squares with a single hole.

Since realizations are incomplete latin squares, we list here a few important results for realizations from Heinrich \cite{heinrich1982disjoint,heinrich2006latin}, D{\'e}nes and P{\'a}sztor \cite{denes1963some}, and Kuhl and Schroeder \cite{kuhl2018latin}. These concern realizations with few subsquares or with subsquares of at most two orders.

\begin{theorem*}[\cite{heinrich2006latin}]
\label{thm: small k squares}
    Take a partition $(h_1h_2\dots h_k)$ of $n$ with $h_1\geq h_2\geq \dots\geq h_k > 0$. Then an $\LS(h_1h_2\dots h_k)$
    \begin{itemize}
        \item always exists when $k=1$;
        \item never exists when $k=2$;
        \item exists when $k=3$ if and only if $h_1=h_2=h_3$;
        \item exists when $k=4$ if and only if $h_1=h_2=h_3$, or $h_2=h_3=h_4$ with $h_1\leq 2h_4$.
    \end{itemize}
\end{theorem*}

\begin{theorem*}[\cite{denes1963some}]
\label{thm: a^k square}
    For $k\geq 1$ and $a\geq 1$, an $\LS(a^k)$ exists if and only if $k\neq 2$.
\end{theorem*}

\begin{theorem*}[\cite{heinrich1982disjoint,kuhl2018latin}]
\label{thm: squaresatmost2}
    For $a>b>0$ and $k>4$, an $\LS(a^ub^{k-u})$ exists if and only if $u\geq 3$, or $0< u < 3$ and $a\leq (k-2)b$.
\end{theorem*}

For incomplete latin squares, we make progress in a similar way to realizations. In \Cref{sec: necessary cond} we prove a necessary condition for ILS that is similar to the best known condition for realizations. The results in \Cref{sec: small cases} determine existence for ILS with at most three subsquares or with subsquares of only a single order. In \Cref{sec: general con} we give a general construction for all incomplete latin squares with disjoint subsquares of orders $h_1,\dots,h_k$ where $n\geq h_1+\sum_{i=1}^kh_i$.

\section{Outline squares and frequency arrays}

In this section we define two important tools that are used throughout: outline rectangles and frequency arrays.

\begin{definition}
    Given partitions $P,Q,R$ of $n$, where $P = (p_1\dots p_u)$, $Q = (q_1\dots q_v)$, $R = (r_1\dots r_t)$, let $O$ be a $u\times v$ array of multisets, with elements from $[t]$. For $i\in [u]$ and $j\in[v]$, let $O(i,j)$ be the multiset of symbols in cell $(i,j)$ and let $|O(i,j)|$ be the number of symbols in the cell, including repetition.

    Then $O$ is an \emph{outline rectangle} associated to $(P,Q,R)$ if
    \begin{enumerate}[label=\textup{(\arabic*)}]
        \item $|O(i,j)| = p_iq_j$, for all $(i,j)\in[u]\times[v]$;
        \item symbol $\ell\in[t]$ occurs $p_ir_\ell$ times in the row $(i,[v])$;
        \item symbol $\ell\in[t]$ occurs $q_jr_\ell$ times in the column $([u],j)$.
    \end{enumerate}
\end{definition}

If $P=Q=R$, then we say that $O$ is an \emph{outline square}.

The array of multisets in \Cref{fig: outline rectangle example} is an outline square associated to $(3^1 2^1 1^3)$.

\begin{figure}[H]
    \centering
    $\arraycolsep=4pt\begin{array}{|ccc|cc|c|c|c|}\hline
    1 & 1 & 1 & 3 & 3 & 2 & 2 & 2\\ 
    1 & 1 & 1 & 4 & 4 & 2 & 2 & 2\\ 
    1 & 1 & 1 & 5 & 5 & 5 & 3 & 4\\ \hline
    3 & 3 & 4 & 2 & 2 & 1 & 1 & 1\\ 
    4 & 5 & 5 & 2 & 2 & 1 & 1 & 1\\ \hline
    2 & 2 & 5 & 1 & 1 & 3 & 4 & 1\\ \hline
    2 & 2 & 3 & 1 & 1 & 4 & 1 & 5\\ \hline
    2 & 2 & 4 & 1 & 1 & 1 & 5 & 3\\ \hline
    \end{array}$
    \caption{An outline square associated to $(3^1 2^1 1^3)$.}
    \label{fig: outline rectangle example}
\end{figure}

Hilton \cite{hilton1980reconstruction} introduced outline rectangles. It is simple to obtain outline rectangles from latin squares.

\begin{definition}
    Given partitions $P,Q,R$ of $n$, where $P = (p_1\dots p_u)$, $Q = (q_1\dots q_v)$ and $R = (r_1\dots r_t)$, and a latin square $L$ of order $n$, the \emph{reduction modulo $(P,Q,R)$} of $L$, denoted $O$, is the $u\times v$ array of multisets obtained by amalgamating rows $(p_1 + \dots + p_{i-1}) + [p_i]$ for all $i\in[u]$, columns $(q_1 + \dots + q_{j-1}) + [q_j]$ for all $j\in[v]$, and symbols $(r_1 + \dots + r_{\ell-1}) + [r_\ell]$ for all $\ell\in[t]$.

    When amalgamating symbols, for $\ell\in[t]$ we map all symbols in $(r_1 + \dots + r_{\ell-1}) + [r_\ell]$ to symbol $\ell$.
\end{definition}

The outline square in \Cref{fig: outline rectangle example} is a reduction modulo $(P,P,P)$, for $P = (3^1 2^1 1^3)$, of the latin square given in \Cref{fig: ILS example}.

It is clear that if an array $O$ is a reduction modulo $(P,Q,R)$ of a latin square, then $O$ is an outline rectangle associated to $(P,Q,R)$. If we instead start with an outline rectangle $O$, and there exists a latin square $L$ such that $O$ is a reduction modulo $(P,Q,R)$ of $L$, then we say that $O$ \emph{lifts} to $L$. Hilton \cite{hilton1980reconstruction} has shown that this reverse operation is possible for all outline rectangles.

\begin{theorem}[Theorem 3 of \cite{hilton1980reconstruction}]
\label{thm: outline rectangle to square}
    Let $P,Q,R$ be partitions of $n$. For every outline rectangle $O$ associated to $(P,Q,R)$, there is a latin square $L$ of order $n$ such that $O$ lifts to $L$.
\end{theorem}

By fixing certain multisets to be the subsquares, we can force an outline square to lift to a latin square with disjoint subsquares.

\begin{lemma}
\label{lemma: ILS from outline}
    For a partition $P = (h_1\dots h_{k+1})$ of $n$, an $\ILS(n;h_1\dots h_k)$ exists if and only if there exists an outline square $O$ associated to $P$ with $O_i(i,i) = h_i^2$ for all $i\in[k]$.
\end{lemma}

For $i,j,\ell\in[u]$, $O_\ell(i,j)$ denotes the number of copies of symbol $\ell$ in the multiset $O(i,j)$.

A \emph{rational outline square} is similar to an outline square, except that the number of copies of a symbol in a cell is a non-negative rational number.

\begin{definition}
    For a partition $P$ of $n$, where $P = (p_1\dots p_u)$, let $O_\ell(i,j)$ be a non-negative rational number for all $i,j,\ell\in[u]$.

    Then the values $O_\ell(i,j)$ form a \emph{rational outline square} associated to $P$ if
    \begin{enumerate}[label=\textup{(\arabic*)}]
        \item $\sum_{\ell\in[t]} O_\ell(i,j) = p_ip_j$, for all $i,j\in[u]$;
        \item $\sum_{j\in[u]} O_\ell(i,j) = p_ip_\ell$, for all $i,\ell\in[u]$;
        \item $\sum_{i\in[u]} O_\ell(i,j) = p_jp_\ell$, for all $j,\ell\in[u]$.
    \end{enumerate}
\end{definition}

Given a partition $P = (h_1\dots h_{k+1})$ and a set $S\subseteq[k+1]$, a (rational) outline square $O$ is said to \emph{respect} $(P,S)$ if $O$ is associated to $P$ and $O_i(i,i) = h_i^2$ for all $i\in S$.

We now add an extra property to outline squares.

\begin{definition}
    An outline square $O$, is \emph{symmetric} if $O_\ell(i,j) = O_c(a,b)$ for every permutation $(a,b,c)$ of $(i,j,\ell)$.
\end{definition}

For a symmetric outline square, the ordering of the arguments $i,j,\ell$ in $O_\ell(i,j)$ is not important, and so we use $O(i,j,\ell)$ to represent the value of $O_c(a,b)$, where $(a,b,c)$ is any permutation of $(i,j,\ell)$.

Let $\ROS(P,S)$ denote a symmetric rational outline square respecting $(P,S)$.

The other tool we use throughout the constructions are frequency arrays, which were introduced in \cite{kemp2025further} and are similar to outline squares.

\begin{definition}
    A \emph{frequency array} $F$ of order $k$ is a $k\times k$ array, where each cell contains a single non-negative integer.
\end{definition}

\begin{definition}
    Let $O$ be a $k\times k$ array of multisets. $O(i,j)$ denotes the multiset of symbols in cell $(i,j)$, $O_{\ell}^i$ and $^jO_\ell$ denote the number of copies of symbol $\ell$ in row $i$ and column $j$ respectively.
    Then $O$ is an \emph{outline array} corresponding to a frequency array $F$ of order $k$, if
    \begin{itemize}
        \item $|O(i,j)| = F(i,j)$,
        \item $O_\ell^i = F(i,\ell)$, and
        \item $^jO_\ell = F(\ell,j)$.
    \end{itemize}
\end{definition}

An outline square associated to $P = (h_1\dots h_k)$ is equivalent to an outline array for a frequency array of order $k$ with $F(i,j) = h_ih_j$.

We state two results here that help to construct outline arrays. The proofs of \Cref{lemma: sum freq arrays,lemma: summing rows/columns of outline array} can be found in \cite{kemp2025further}.

\begin{lemma}[Lemma 40 of \cite{kemp2025further}]
\label{lemma: sum freq arrays}
    If $O_1$ and $O_2$ are outline arrays corresponding to the frequency arrays $F_1$ and $F_2$ respectively, then there exists an outline array $O^*$ corresponding to the frequency array $F^*$ where $F^*(i,j) = F_1(i,j) + F_2(i,j)$.
\end{lemma}

Observe that an outline square respecting $(P,[k])$, where $P = (h_1\dots h_{k+1})$, is equivalent to the sum of outline arrays $O_1$ and $O_2$, corresponding to $F_1$ and $F_2$, where $F_1(i,j) = h_ih_j$ for all $i\neq j$ or $i=j=k+1$, $F_1(i,i) = 0$ for $i\in[k]$, $F_2(i,j) = 0$ when $i\neq j$ or $i=j=k+1$ and $F_2(i,i) = h_i^2$ for $i\in[k]$.

\begin{lemma}[Lemma 41 of \cite{kemp2025further}]
\label{lemma: summing rows/columns of outline array}
    If an outline array $O$ exists for an order $k$ frequency array $F$, then for any partition $S_1,S_2,\dots,S_{k'}$ of $[k]$, an outline array $O^*$ exists for the order $k'$ array $F^*$, where for all $i,j\in[k']$ $$F^*(i,j) = \sum_{x\in S_i}\sum_{y\in S_j}F(x,y).$$
\end{lemma}

\section{Necessary conditions for ILS}
\label{sec: necessary cond}

Since the orders of the $k$ subsquares do not necessarily partition $n$, we will represent the orders as a partition $P = (h_1\dots h_kh_{k+1})$ of $n$ where $h_{k+1} = n-\sum_{i=1}^kh_i$. Using a partition of $n$ allows us to easily reuse many of the same methods that have been used for realizations.

The best known necessary condition for existence of a realization is given in \cite{kemp2025further}. That result relies on rational outline squares for realizations and we prove corresponding results here for incomplete latin squares.

\begin{lemma}
\label{lemma: always symmetric solution ILS}
    If an $\ILS(n;h_1\dots h_k)$ exists, then an $\ROS(h_1\dots h_{k+1},[k])$ exists.
\end{lemma}
\begin{proof}
    We assume that the subsquares are along the main diagonal of the incomplete latin square such that for all $\ell\in[k]$ and all $i,j\in \sum_{m=1}^{\ell-1}h_m+[h_\ell]$, the entry in cell $(i,j)$ is also in $\sum_{m=1}^{\ell-1}h_m+[h_\ell]$. Obtain an outline square $O$ which respects $(P,[k])$ for $P = (h_1\dots h_{k+1})$ by reducing an $\ILS(n;h_1\dots\linebreak[1] h_k)$ modulo $P$. For any multiset $i,j,\ell\in[k+1]$, let
    $$X(i,j,\ell) = \frac{1}{6}\big(O_\ell(i,j) + O_\ell(j,i) + O_i(j,\ell) + O_i(\ell,j) + O_j(i,\ell) + O_j(\ell,i)\big).$$

    Clearly, $X(i,i,i) = \frac{1}{6}(6\cdot O_i(i,i)) = h_i^2$ for any $i\in[k]$ as required to respect the partition, and by fixing any $i,j\in[k+1]$, we have that
    \begin{align*}
        \begin{split}
        \sum_{\ell\in[k+1]} X(i,j,\ell) &= \frac{1}{6}\bigg( \sum_{\ell=1}^{k+1} O_\ell(i,j) + \sum_{\ell=1}^{k+1} O_\ell(j,i) + \sum_{\ell=1}^{k+1} O_i(j,\ell) + \sum_{\ell=1}^{k+1} O_i(\ell,j) \\
        &\qquad +\sum_{\ell=1}^{k+1} O_j(i,\ell) + \sum_{\ell=1}^{k+1} O_j(\ell,i) \bigg)\\
        &= \frac{1}{6}(6h_ih_j) = h_ih_j.
        \end{split}
    \end{align*}

    Thus, $X = \{X(i,j,\ell)\mid \{i,j,\ell\} \text{ is a multiset of } [k+1]\}$ forms an $\ROS(h_1\dots h_k,[k])$.
\end{proof}

There are no known necessary conditions for ILS with an arbitrary number of subsquares. The following result gives a similar condition to that given for realizations.

\begin{theorem}
\label{cond: sufiĉa? ILS}
    If an $\ILS(n;h_1\dots h_k)$ exists, then
    $$\left(\sum_{i\in A\cup C}h_i\right)^2 + \left(\sum_{i\in B\cup D}h_i\right)^2 - \sum_{i\in E\setminus\{k+1\}}h_i^2 \geq \left(\sum_{i\in A\cup D}h_i\right)\left(\sum_{i\in B\cup C}h_i - \sum_{j\in \overline{E}}h_j\right),$$
    where $A$, $B$, $C$ and $D$ are pairwise disjoint subsets of $[k+1]$, $E = A\cup B\cup C\cup D$ and $\overline{E} = [k+1]\setminus E$.
\end{theorem}
\begin{proof}
    Using \Cref{lemma: always symmetric solution ILS}, follow the same argument as in the proof of Theorem 20 from \cite{kemp2025further}. It is stated there that for a set $X\in\{A,B,C,D\}$, the number of symbols from $X$ in the rows or columns of $X$ is at most $(\sum_{i\in X} h_i)^2-\sum_{i\in X}h_i^2$ due to the subsquares. This holds unless $k+1\in X$, where the number of symbols is at most $(\sum_{i\in X} h_i)^2-\sum_{i\in X\setminus\{k+1\}}h_i^2$. The rest of the proof is unchanged and this provides the required inequality.
\end{proof}

\section[Constructions for small cases]{Constructions for small cases}
\label{sec: small cases}

In the previous sections we showed that outline squares can be used to determine the existence of incomplete latin squares. As with realizations, finding such an outline square is done by finding values for each variable $O_\ell(i,j)$ where $i,j,\ell\in[k+1]$. This is simple when $k$ is very small. When $k=1$, an incomplete latin square is equivalent to a latin square of order $n$ with a single subsquare of order $h_1$. We know from \Cref{thm: single subsquare} that this exists if and only if $h_1\leq \frac{n}{2}$ (or trivially $h_1=n$). Thus, we consider $k=2$ and $k=3$.

\begin{theorem}
\label{thm: ILS(n;h1h2)}
    For any partition $(h_1h_2h_3)$ of $n$ where $h_1\geq h_2>0$, there exists an $\ILS(n;h_1h_2)$ if and only if $h_3\geq h_1$.
\end{theorem}
\begin{proof}
By \Cref{lemma: ILS from outline}, we need only find an outline square $O$ associated to $P = (h_1h_2h_3)$ with $O_1(1,1) = h_1^2$ and $O_2(2,2) = h_2^2$. This forces $O_1(1,2) = O_2(1,2) = 0$, and since $\sum_{\ell=1}^3O_\ell(1,2) = h_1h_2$, it is clear that $O_3(1,2) = h_1h_2$. Since $\sum_{\ell=1}^3O_3(\ell,2) = h_3h_2$, it follows that $h_1\leq h_3$.

Consider the outline square as shown in \Cref{fig: ILS(h1h2) outline}, where $x:y$ in cell $(i,j)$ represents $O_x(i,j) = y$.

    \begin{figure}[H]
        \centering
        \renewcommand{\arraystretch}{2.7}
        \begin{tabular}{|l|l|l|}\hline
        $1:h_1^2$ & $3:h_1h_2$ & \thead{$2:h_1h_2$\\$3:h_1h_3 - h_1h_2$} \\\hline
        $3: h_1h_2$ & $2:h_2^2$ & \thead{$1:h_1h_2$\\$3:h_2h_3-h_1h_2$} \\\hline
        \thead{$2:h_1h_2$\\$3:h_1h_3 - h_1h_2$} & \thead{$1:h_1h_2$\\$3:h_2h_3-h_1h_2$} & \thead{$1:h_1h_3-h_1h_2$\\$2:h_2h_3-h_1h_2$\\$3:h_3^2 - h_1h_3 - h_2h_3 +2h_1h_2$} \\\hline
        \end{tabular}
        \caption{The outline square for an $\ILS(n;h_1h_2)$}
        \label{fig: ILS(h1h2) outline}
    \end{figure}

The above array is a valid outline square when all values are non-negative. Since $h^3 - h_1h_3 - h_2h_3 + 2h_1h2 = (h_3-h_1)(h_3-h_2)+h_1h_2$, this is true when $h_3\geq h_1\geq h_2$.
    
Therefore, the outline square is valid when $h_3\geq h_1$.
\end{proof}

\begin{theorem}
\label{thm: ILS(n;h1h2h3)}
    For any partition $(h_1h_2h_3h_4)$ of $n$ where $h_1\geq h_2\geq h_3$, an $\ILS(n;h_1h_2h_3)$
    \begin{itemize}
        \item exists when $h_4\geq h_1$;
        \item exists when $h_1>h_4\geq h_3$ if and only if $h_4\geq h_1-h_3$; and
        \item exists when $h_4<h_3$ if and only if $h_4\geq h_2+h_3-2\frac{h_2h_3}{h_1}$ and $h_4^2+h_4(2h_1-h_2-h_3)-h_1h_2-h_1h_3+2h_2h_3\geq 0$.
    \end{itemize}
\end{theorem}
\begin{proof}
    In each case the necessary conditions come from \Cref{cond: sufiĉa? ILS}. Setting $A = \{2\}$, $B = \{1\}$, $C = \{3\}$ and $D = \emptyset$ gives that $h_4\geq h_1-h_3$. Let $A = \{1\}$, $B = \{2,3\}$ and $C = D = \emptyset$ to get $h_4\geq h_2 + h_3 - 2\frac{h_2h_3}{h_1}$, and let $A = \{1,4\}$, $B = \{2,3\}$ and $C = D = \emptyset$ to get the final condition.

    Consider the outline square given in \Cref{fig: ILS(h1h2h3) outline}, where $x:y$ in cell $(i,j)$ represents $O_x(i,j) = y$. Since $h_1\geq h_2\geq h_3$, this is a valid outline square when there exists an integer $z\geq 0$ such that
    \begin{enumerate}[label=(\arabic*)]
        \item $z\leq h_2h_3$,
        \item $z\leq h_1h_4-h_1h_2-h_1h_3+2h_2h_3$,
        \item $z\leq h_2h_4-h_1h_2+h_2h_3$,
        \item $z\leq h_3h_4-h_1h_3+h_2h_3$, and
        \item $h_4^2 - h_4(h_1+h_2+h_3)+2h_1h_2+ 2h_1h_3-4h_2h_3+3z\geq 0$.
    \end{enumerate}

    Given an appropriate value of $z$, this outline square satisfies the conditions of \Cref{lemma: ILS from outline}.

    \begin{figure}[h]
        \centering
        \renewcommand{\arraystretch}{2.7}
        \begin{tabular}{|l|l|l|l|}\hline
        $1:h_1^2$ & \thead{$3:h_2h_3$\\$4: h_1h_2-h_2h_3$} & \thead{$2:h_2h_3-z$\\$4:h_1h_3 - h_2h_3+z$} & \thead{$2:h_1h_2-h_2h_3+z$\\$3:h_1h_3 - h_2h_3$\\$4:h_1h_4-h_1h_2-h_1h_3$\\$\qquad+2h_2h_3-z$} \\\hline
        \thead{$3:h_2h_3-z$\\$4: h_1h_2-h_2h_3+z$} & $2:h_2^2$ & $1:h_2h_3$ & \thead{$1:h_1h_2-h_2h_3$\\$3:z$\\$4:h_2h_4-h_1h_2+h_2h_3-z$} \\\hline
        \thead{$2:h_2h_3$\\$4:h_1h_3 - h_2h_3$} & \thead{$1:h_2h_3-z$\\$4:z$} & $3:h_3^2$ & \thead{$1:h_1h_3-h_2h_3+z$\\$4:h_3h_4-h_1h_3+h_2h_3-z$} \\\hline
        \thead{$2:h_1h_2-h_2h_3$\\$3:h_1h_3-h_2h_3+z$\\$4:h_1h_4-h_1h_2-h_1h_3$\\$\qquad+2h_2h_3-z$} & \thead{$1:h_1h_2-h_2h_3+z$\\$4:h_2h_4-h_1h_2$\\$\qquad+h_2h_3-z$} & \thead{$1:h_1h_3 - h_2h_3$\\$2:z$\\$4:h_3h_4-h_1h_3$\\$\qquad+h_2h_3-z$} & \thead{$1:h_1h_4-h_1h_2-h_1h_3$\\$\qquad+2h_2h_3-z$\\$2:h_2h_4-h_1h_2+h_2h_3-z$\\$3:h_3h_4-h_1h_3+h_2h_3-z$\\$4:h_4^2 - h_4(h_1+h_2+h_3)$\\ $\qquad +2h_1h_2+ 2h_1h_3$\\$\qquad-4h_2h_3+3z$}\\\hline
        \end{tabular}
        \caption{An outline square for an $\ILS(n;h_1h_2h_3)$}
        \label{fig: ILS(h1h2h3) outline}
    \end{figure}

    We make different choices for $z$ in three cases.

    \noindent\textbf{Case 1: $h_4\geq h_1$}\\
    Let $z=h_2h_3$. As shown in \Cref{thm: ILS(n;h1h2)}, $h_1h_4-h_1h_2-h_1h_3+2h_2h_3\geq h_2h_3$ when $h_4\geq h_1$. Also, $h_2(h_4-h_1)+h_2h_3\geq h_3(h_4-h_1)+h_2h_3\geq h_2h_3$. Thus, we only show that (5) is satisfied.

    Let $h_4 = h_1+x$ for some $x\geq 0$, and so
    $$h_4^2 - h_4(h_1+h_2+h_3) +2h_1h_2+2h_1h_3-h_2h_3 = x(h_1-h_2) + x^2 + h_1(h_2+h_3) - xh_3 - h_2h_3.$$

    If $x\geq h_3$ then $x^2+h_1h_2\geq xh_3 + h_2h_3$, otherwise $h_1h_2 + h_1h_3\geq xh_3 + h_2h_3$. In either case, $h_4^2 - h_4(h_1+h_2+h_3) +2h_1h_2+2h_1h_3-h_2h_3\geq 0$.

    \noindent\textbf{Case 2: $h_1>h_4\geq h_3$}\\
    Set $z = h_2h_4 - h_1h_2 + h_2h_3$ and assume that $h_4\geq h_1-h_3$. Since $h_4<h_1$, $z = h_2(h_4-h_1)+h_2h_3\leq h_3(h_4-h_1)+h_2h_3\leq h_2h_3$ and $0\leq h_2(h_4-(h_1-h_3)) = z$ from $h_4\geq h_1-h_3$. Also,
    \begin{align*}
        h_1h_4 - h_1h_2 - h_1h_3 + 2h_2h_3 - z &= h_1h_4 - h_2h_4 - h_1h_3 + h_2h_3\\
        &= (h_1-h_2)(h_4-h_3)
    \end{align*}
    and so (2) is satisfied since $h_1\geq h_2$ and $h_4\geq h_3$.

    Letting $h_4 = h_1-h_3+x$ for some $x\geq 0$,
    \begin{align*}
        h_4^2 - h_4(h_1+h_2+h_3)+2h_1h_2+ 2h_1h_3-4h_2h_3+3z &= x^2 + (h_1x + 2h_2x - 3h_3x)\\
        &\qquad +~(h_1^2 + h_1h_2 + 2h_3^2 - 3h_2h_3)
    \end{align*}
    and observe that
    \begin{align*}
        h_1^2 + h_1h_2 + 2h_3^2 - 3h_2h_3 &\geq 2h_2^2 + 2h_3^2 - 3h_2h_3\\
        &= 2(h_2^2 + h_3^2 - 2h_2h_3) + h_2h_3\\
        &= 2(h_2-h_3)^2+h_2h_3\\
        &\geq 0.
    \end{align*}
    Therefore, (5) is satisfied.

    \noindent\textbf{Case 3: $h_4< h_3$}\\
    Let $z = h_1h_4-h_1h_2-h_1h_3+2h_2h_3$ and assume that $h_4\geq h_2+h_3-2\frac{h_2h_3}{h_1}$ and $h_4^2+h_4(2h_1-h_2-h_3)-h_1h_2-h_1h_3+2h_2h_3\geq 0$.

    Thus, $z = h_1(h_4 - h_2-h_3+2\frac{h_2h_3}{h_1})\geq 0$ and as in case 2, $h_2(h_4-h_1)+h_2h_3\leq h_3(h_4-h_1)+h_2h_3\leq h_2h_3$. Since $h_4<h_3$ here and $h_2h_4-h_1h_2+h_2h_3-z = (h_1-h_2)(h_3-h_4)$, (3) is satisfied.

    Observe that
    \begin{align*}
        h_4^2 - h_4(h_1+h_2+h_3)+2h_1h_2+ 2h_1h_3-4h_2h_3+3z &= h_4^2+h_4(2h_1-h_2-h_3)\\
        &\qquad-~h_1h_2-h_1h_3+2h_2h_3,
    \end{align*}
    and so (5) is satisfied by assumption.

    Therefore, in all cases there exists an outline square when the necessary conditions are met.
\end{proof}

The other small case considered for realizations is partitions where the number of distinct integers in the partition is limited. We use the same approach here, but allow $h_{k+1}$ to be distinct to the other parts.

\begin{theorem}
\label{thm: ILS(n;h1^k)}
    An $\ILS(n;h_1^k)$ exists for
    \begin{itemize}
        \item $k=1$ if and only if $n=h_1$ or $n\geq 2h_1$, 
        \item $k=2$ if and only if $n\geq 3h_1$,
        \item $k\geq 3$ if and only if $n\geq kh_1$.
    \end{itemize}
\end{theorem}
\begin{proof}
    As discussed earlier, the case where $k=1$ is covered by \Cref{thm: single subsquare}, and so an $\ILS(n;h_1)$ exists if and only if $n=h_1$ or $n\geq 2h_1$.
    
    For $k=2$, we use the result of \Cref{thm: ILS(n;h1h2)}, and conclude that an $\ILS(n;h_1^2)$ exists if and only if $n\geq 3h_1$.

    Finally, if $k\geq 3$, then $n = kh_1 + h_{k+1}$.
    Let $h_{k+1}=mh_1 + r$ for some integers $m$ and $r$ with $0\leq r<h_1$. Thus, an $\LS(h_1^{k+m}r)$ is an $\ILS(n;h_1^k)$. Since $k\geq 3$, we have that $k+m\geq 3$ and $r<h_1$, so this realization exists by \Cref{thm: small k squares,thm: a^k square,thm: squaresatmost2}. Therefore, an $\ILS(n;h_1^k)$ exists for all $n\geq kh_1$.
\end{proof}

\section{A general construction for incomplete latin squares}
\label{sec: general con}

Complete results for partitions with a limited number of parts or a limited number of distinct parts are useful and follow the approach taken for realizations. However, it is most interesting to have results which hold for any choice of orders for the subsquares. When considering incomplete latin squares, the order of the latin square must be at least the sum of the orders of the subsquares. It is natural to ask how close $n$ can be to this sum. Formally, given a partition $(h_2\dots h_k)$, what is the smallest value of $h_1$ such that an $\ILS(n;h_2\dots h_k)$ exists for all $n\geq h_1+\sum_{i=2}^kh_i$?

It was shown in \cite{kemp2025threedisjoint} that a realization always exists when the largest three subsquares are of the same order.

\begin{theorem*}[\cite{kemp2025threedisjoint}]
\label{thm: 3 subsquares}
    There exists an $\LS(h_1^3h_4\dots h_k)$ for all $h_1\geq h_4\geq\dots\geq h_k$.
\end{theorem*}

Using this result on realizations, we can easily get an upper bound for this smallest possible $h_1$.

\begin{theorem}
\label{thm: ILS 2h1}
    Given an integer partition $(h_2\dots h_k)$, if $h_1\geq 2h_2$ and $n = h_1 + \sum_{i=2}^kh_i$, then there exists an $\ILS(n;h_2\dots h_k)$.
\end{theorem}
\begin{proof}
    Let $n-2h_2-\sum_{i=2}^kh_i = qh_k + r$, where $q\geq 0$ and $0\leq r< h_k$. By \Cref{thm: 3 subsquares}, there exists an $\LS(h_2^3h_3\dots h_k^{q+1}r)$. This is an $\ILS(n;h_2\dots h_k)$.
\end{proof}

By a similar construction to \Cref{thm: 3 subsquares} we can get lower values of $h_1$ for all partitions. We start with a result used in that construction which only holds for values of $r$ that are odd, where $r = \sum_{i=2}^kh_i$. We then obtain the same result for even $r$.

\begin{lemma*}[\cite{kemp2025threedisjoint}]
\label{circulant_construction}
    Let $(h_1h_2h_3\dots h_k)$ be an integer partition where $r=\sum_{i=2}^k h_i$ is odd and $h_i\ge h_{i+1}$ for $2\le i \le k-1$. Further suppose that $h_2 \le (r+1)/4$ and $2h_2 \le h_1 \le r+1-2h_2$.
    Then there exists an $\LS(h_1h_2\dots h_k)$.
\end{lemma*}

Throughout the following proof we work modulo $r$ with residues in $[r]$. We use $a\oplus b$ and $a\ominus b$ to denote $a+b\pmod r$ and $a-b\pmod r$ respectively. Also, let $k\{x\}$ denote the multiset consisting of $k$ copies of element $x$, and so $\sum_{i=1}^n k_i\{x_i\} $ is the multiset consisting of $k_i$ copies of $x_i$ for $i\in [n]$.

\begin{lemma}
\label{lemma: circulant square for even r}
    For integers $r$, $h_1$ and $h_2$, where $r$ is even, $1\leq h_2\leq \frac{1}{4}r$ and $2h_2\leq h_1\leq r+1-2h_2$, there exists a partial latin square $L$ of order $r$ with $r-h_1$ transversals such that
    \begin{itemize}
        \item the cells $(i,j)$ where $j\ominus i\in\{2\ell-1\mid \ell\in[h_2]\}\cup\{r+1-2\ell\mid \ell\in[h_2]\}$ are empty, and
        \item the cells $(i,j)$ where $j\ominus i\in\{2\ell\mid\ell\in[h_2-1]\}\cup\{r+2-2\ell\mid\ell\in[h_2]\}$ form transversals where for $0\leq d\leq h_2-1$, $L(i,i\oplus 2d) = L(i\oplus 2d,i) = i\oplus d$.
    \end{itemize}
\end{lemma}
\begin{proof}
    There are $h_1\geq 2h_2$ empty cells in each row and column of $L$, where $2h_2$ of those are fixed by the requirements of the lemma. We construct $L$ to have the $2h_2$ required empty cells per row, and more can be gained by removing any of the $r-4h_2+1$ transversals not required in the lemma.

    We start by completing the first row of $L$. For $i\in[r]$, let
    $$L(1,i) = \begin{cases}
        \emptyset, & \text{if $i=2j$ or $i=r+2-2j$ for $j\in[h_2]$,}\\
        \frac{i+1}{2}, & \text{if $i=2j-1$ for $j\in[h_2]$,}\\
        \frac{i+r+1}{2}, & \text{if $i=r+1-2j$ for $j\in[h_2-1]$,}\\
        r-i+h_2+2, & \text{if $i=2h_2+j$ for $j\in[\frac{r-4h_2+2}{2}]$,}\\
        r-i-h_2+3, & \text{if $i=\frac{r+2}{2}+j$ for $j\in[\frac{r-4h_2}{2}]$.}
    \end{cases}$$
    The four cases with non-empty $L(1,i)$ respectively give the all symbols $s$ such that $1\leq s\leq h_2$, $r+2-h_2\leq s\leq r$, $\frac{r}{2}+h_2+1\leq s\leq r-h_2+1$ and $h_2+2\leq s\leq \frac{r}{2}-h_2+1$, and these are clearly four distinct ranges. Thus, there are $r-2h_2$ unique symbols in the first row of $L$.

    For each $i,j\in[r]$ let
    $$L(i,j) = L(1,j\ominus i\oplus 1)\oplus (i-1).$$
    It follows that each row is a cyclic shift of the columns and symbols of the first row. Thus, each row has no repeated symbols.

    If $L(i,j) = L(i',j)$ for some $i,i',j\in[r]$, then
    \begin{align*}
        L(1,j\ominus i\oplus 1)\oplus (i-1) &= L(1,j\ominus i'\oplus 1)\oplus (i'-1)\\
        L(1,j\ominus i\oplus 1) \ominus L(1,j\ominus i'\oplus 1) &= i'\ominus i\\
        L(1,a) \ominus L(1,b) &= a\ominus b
    \end{align*}
    for $a=j\ominus i\oplus 1$ and $b=j\ominus i'\oplus 1$.

    Suppose that $a>b$ and $L(1,a) \ominus L(1,b)= a-b$. If $a,b\in\{2j-1\mid j\in[h_2]\}$ then $\frac{a+1}{2}-\frac{b+1}{2} = \frac{a-b}{2} = a-b$. This is only true when $a=b$. If $a,b\in\{r+1-2j\mid j\in[h_2-1]\}$, then $\frac{a-b}{2}=a-b$ again and so $a=b$.

    If $2h_2+1\leq a,b\leq \frac{r+2}{2}$, then $-a+b \equiv a-b$. This holds only when $a-b\equiv \frac{r}{2},r$ and since $a-b\leq \frac{r-4h_2}{2}$, it must be that $a=b$. Similarly, if $\frac{r+4}{2}\leq a,b\leq r+1-2h_2$ then $-a+b\equiv a-b$. Thus, $a=b$ in this case also.

    We also check each combination of the four cases.

    If $a\in\{r+1-2j\mid j\in[h_2-1]\}$ and $b\in\{2j-1\mid j\in[h_2]\}$ then $\frac{r}{2} = \frac{a-b}{2}$, thus $a=b$.

    If $b\in\{2j-1\mid j\in[h_2]\}$ then there exists $j\in[h_2]$ such that $b=2j-1$ and $L(1,b) = j$. If $2h_2+1\leq a\leq \frac{r+2}{2}$ then $r-a+h_2+2 - j \equiv a-2j+1$ and so $j \equiv 2a-h_2-1$. Thus, for a contradiction, $3h_2+1\leq j\leq r-h_2+1$. Similarly, if $\frac{r+4}{2}\leq a\leq r+1-2h_2$ then $j\equiv 2a+h_2-2$ and so $h_2+2\leq j\leq r-3h_2$.
    If $a\in\{r+1-2j\mid j\in[h_2-1]\}$ then there exists $j\in[h_2-1]$ such that $a=r+1-2j$ and $L(1,a) = r+1-j$. If $2h_2+1\leq b\leq \frac{r+2}{2}$ then $r+1-j-r+b-h_2-2 \equiv r+1-2j-b$. It follows that $j \equiv -2b+h_2+2$ and so $h_2\leq j\leq r-3h_2$. Similarly, if $\frac{r+4}{2}\leq b\leq r+1-2h_2$ then $j\equiv -2b-h_2+3$ and so $3h_2+1\leq j\leq r-h_2-1$. In either case, this is not possible.

    Finally, if $\frac{r+4}{2}\leq a\leq r+1-2h_2$ and $2h_2+1\leq b\leq \frac{r+2}{2}$, then $r-a-h_2+3-r+b-h_2-2 \equiv a-b$. Thus, $-a+b-2h_2+1 \equiv a-b$ which means that $2(a-b)\equiv -2h_2+1$. Since $r$ is even, this is a contradiction.

    Therefore, $L(1,a)-L(1,b)\not\equiv a-b$ for all $a,b\in[r]$ and so $L$ is a partial latin square.
\end{proof}

\begin{lemma}
\label{even r construction}
    Let $(h_1h_2h_3\dots h_k)$ be an integer partition, where $r=\sum_{i=2}^k h_i$ is even and $h_i\ge h_{i+1}$ for $2\le i \le k-1$. Further suppose that $h_2 \le \frac{1}{4}r$ and $2h_2 \le h_1 \le r+1-2h_2$.
    Then there exists an $\LS(h_1h_2\dots h_k)$.
\end{lemma}
\begin{proof}
    We construct an outline rectangle $O$ associated to $(Q,Q,P)$ where $Q = (h_1 1^r)$.

    Take the partial latin square $L$ constructed in \Cref{lemma: circulant square for even r} and place $L$ across the cells $(i,j)$ for $i,j\in1+[r]$. Amalgamate symbols so that the symbols in $\sum_{j=2}^{i-1}h_j+[h_i]$ become $i$ for all $i\in1+[r]$. Fill the empty cells with symbol $1$.

    For all $i\in1+[r]$, let $S_i = \sum_{j=2}^{i-1}h_j$ and $H_i = S_i + [h_i]$. Observe that for all $a,b\in[h_i]$, if $b\geq a$ and $b-a$ is even, then $b = a+2d$ for some $0\leq d\leq \frac{h_i-1}{2}$ and $L(S_i+a,S_i+b) = L(S_i+b,S_i+a) = S_i + a+d\in H_i$ since $b-a\leq h_i-1\leq h_2-1$. Thus, $O(S_i+1+a,S_i+1+b) = \{i\}$ for all $a,b\in[h_i]$.

    If $b\ominus a$ is instead odd, then $L(S_i+a,S_i+b)$ is empty and $O(S_i+1+a,S_i+1+b) = \{1\}$. If $a<b$, then consider also the entries of $L$ in cells $(S_i+a,y)$, $(x,S_i+b)$ and $(x,y)$ where $x = (S_i+1)\ominus a$ and $y = (S_i + 2h_i +1)\ominus b$. Since $b-a$ is odd and $a<b$, $y\ominus(S_i+a) = 2h_i+1-a-b\leq 2h_i-2$ and is even, $(S_i+b)\ominus x = b+a-1 \leq 2h_i-2$ and is even, and $y\ominus x = 2h_i-(b-a)\leq 2h_1-1$ and is odd. Thus, for $d = \frac{1}{2}(2h_i+1-a-b)\leq h_i-a$ and $a\leq d' = \frac{1}{2}(a+b-1)\leq h_i-1$, $L(S_i+a,y) = L(S_i+a,(S_i+a)\oplus2d) = S_i + a + d\in H_i$, $L(x,S_i+b) = L(x,x\oplus 2d') = S_i+1-a+d\in H_i$ and $L(x,y)$ is empty. Correspondingly, $O(S_i+1+a,y+1) = \{i\}$, $O(x+1,S_i+1+b) = \{i\}$ and $O(x+1,y+1) = \{1\}$.

    Therefore, the set of cells $\{(S_i+1+a,S_i+1+b),(S_i+1+a,y+1),(x+1,S_i+1+b),(x+1,y+1)\}$ form a subsquare in $O$. Swapping the entries of $\{i\}$ and $\{1\}$ makes $O(S_i+1+a,S_i+1+b) = \{i\}$ for all $a,b\in[h_i]$ where $a<b$.

    By repeating this with the cells $\{(S_i+1+b,S_i+1+a),(S_i+1+b,x+1),(y+1,S_i+1+a),(y+1,x+1)\}$, the cells $O(a+1,b+1)$ for all $a,b\in H_i$ contain $\{i\}$.

    Clearly, no cells $(S_i+a,y)$ or $(x,S_i+b)$ are repeated across different values of $i$ since those cells of $L$ contain an element of $H_i$. Also, $y-1 = x\oplus 2d$ for $a\leq d = \frac{1}{2}(2h_i -1+ a -b)\leq h_i-1$ and $x\oplus d\in H_i$, so none of the $(x,y)$ cells are shared between $i$ values either.

    The cells $(i,j)$ of $O$ for $i,j\in1+[r]$ are filled and have the required subsquares. The remaining cells are $(1,1)$, $(i,1)$ and $(1,j)$.

    Take $O(1,1) = h_1^2\{1\}$. For each column $j\in[r]$ of $L$, there are $h_1$ symbols that do not appear in that column (since there were $h_1$ empty cells in each column). Since we have not changed the number of copies of the amalgamated symbols in column $j+1$ of $O$, there are $h_1$ copies of symbols in $[k]\setminus[1]$ that do not appear in column $j+1$ of $O$. Place these $h_1$ copies in cell $(1,j+1)$. Repeat this for the rows to fill the cells $(j+1,1)$. Since each symbol of $[r]$ appears $r-h_1$ times in $L$ (once in each of the $r-h_1$ transversals), there are $h_1h_\ell$ copies of $\ell\in[k]\setminus[1]$ in each of the first row and first column of $O$.

    Therefore, $O$ is a completed outline rectangle for an $\LS(h_1\dots h_k)$.
\end{proof}

To prove the main result of this section we use the realizations constructed above and combine them with outline arrays. The following lemma constructs the required outline arrays.

\begin{lemma}
\label{lemma: ILS L thing}
    For $k\geq m\geq 2$ and $h_{m+1}\geq \dots\geq h_k$, if $(m-1)b\geq c\sum_{i=m+1}^kh_i$ then there exists an outline array for the frequency array $F$ of order $k$ where for all $i,j\in[k]$
    $$F(i,j) = \begin{cases}
        b, & \text{if $i,j\in[m]$ and $i\neq j$ unless $i=j=1$,}\\
        ch_j, & \text{if $i\in[m]$ and $j\geq m+1$,}\\
        ch_i, & \text{if $j\in[m]$ and $i\geq m+1$,}\\
        0, & \text{otherwise.}
    \end{cases}$$
    Furthermore, if $m=5$, $2b\geq a\geq b$ and $4b\geq c\sum_{i=m+1}^kh_i$ then there exists an outline array for the frequency array $F'$ of order $k$ where for all $i,j\in[k]$
    $$F(i,j) = \begin{cases}
        a, & \text{if $i\in[5]$ and $j\in[2]$, or $j\in[5]$ and $i\in[2]$, and $(i,j)\neq (2,2)$,}\\
        b, & \text{if $i,j\in2+[3]$ and $i\neq j$,}\\
        ch_j, & \text{if $i\in[5]$ and $j>5$,}\\
        ch_i, & \text{if $j\in[5]$ and $i>5$,}\\
        0, & \text{otherwise.}
    \end{cases}$$
\end{lemma}
\begin{proof}
    Let $A$ be a multiset consisting of $ch_i$ copies of $i$ for each $i\in [k]\setminus[m]$. Since $|A| = c\sum_{\ell=4}^{k}h_\ell \leq (m-1)b$, partition the set $A$ into multisets $A_1,A_2,\dots,A_b$, where $0\leq |A_\ell| \leq m-1$ for all $\ell\in[b]$.

    Let $A_\ell(s)$ be the number of copies of $s$ in $A_\ell$, for $\ell\in[b]$ and $s\in [k]\setminus[m]$.
    
    For each $\ell\in[b]$, let $F_\ell$ be a frequency array of order $k$ where for all $i,j\in[k]$
    $$F_\ell(i,j) = \begin{cases}
        1, & \text{if $i,j\in[m]$ and $i\neq j$ unless $i=j=1$,}\\
        A_\ell(i), & \text{if $i>m$ and $j\leq m$,}\\
        A_\ell(j), & \text{if $i\leq m$ and $j>m$,}\\
        0, & \text{otherwise.}
    \end{cases}$$

    An outline array corresponding to $F_\ell$ for $\ell\in[b]$ is found by placing a 1 in cell $(1,1)$ of an $\LS(1^m|A_\ell|^1)$, which always exists by \Cref{thm: small k squares,thm: a^k square,thm: squaresatmost2} since $|A_\ell|\leq m-1$ unless $m=2$ and $|A_\ell|=0$. In this last case, let $O_\ell$ be an outline array with $O(1,1) = \{2\}$ and $O(1,2) = O(1,2) = \{1\}$.
    
    Use the partition $\big\{\{i\}\mid i\in[m]\big\}\cup\left\{m+\sum_{j=m+1}^{i-1}A_\ell(j)+[A_\ell(i)]\mid i\in[k]\setminus[m]\right\}$ and \Cref{lemma: summing rows/columns of outline array} to obtain the required outline array $O_\ell$.

    If $m=5$, for $\ell\in[a-b]\subseteq[b]$, instead let
    $$F_\ell(i,j) = \begin{cases}
        2, & \text{if $i\in[5]$ and $j\in[2]$, or $j\in[5]$ and $i\in[2]$, and $(i,j)\neq(2,2)$,}\\
        1, & \text{if $i,j\in2+[3]$ and $i\neq j$,}\\
        A_\ell(i), & \text{if $i>5$ and $j\leq 5$,}\\
        A_\ell(j), & \text{if $i\leq 5$ and $j>5$,}\\
        0, & \text{otherwise.}
    \end{cases}$$

    Use the same partition as above with the appropriate array from \Cref{fig: ILS 5 row bits}.

    \begin{figure}[h]
        \centering
    \begin{subfigure}{0.4\textwidth}
        $$\begin{array}{|c|c|c|c|c|}\hline
            4,5 & 3,5 & 1,2 & 1,2 & 3,4 \\ \hline
            3,4 &  & 4,5 & 3,5 & 1,1 \\ \hline
            1,1 & 4,5 &  & 2 & 2 \\ \hline
            3,5 & 1,1 & 2 &  & 2 \\ \hline
            2,2 & 3,4 & 1 & 1 &  \\ \hline
        \end{array}$$
        \caption{$|A_\ell| = 0$}
    \end{subfigure}
    \begin{subfigure}{0.4\textwidth}
        $$\begin{array}{|c|c|c|c|c|c|}\hline
            4,5 & 4,5 & 1,1 & 3,6 & 2,3 & 2 \\ \hline
            3,3 &  & 4,6 & 1,5 & 1,4 & 5 \\ \hline
            2,2 & 4,5 &  & 1 & 6 & 1 \\ \hline
            5,6 & 1,1 & 2 &  & 2 & 3 \\ \hline
            1,1 & 3,6 & 2 & 2 &  & 4 \\ \hline
            4 & 3 & 5 & 2 & 1 &  \\ \hline
        \end{array}$$
        \caption{$|A_\ell| = 1$}
    \end{subfigure}
    
    \begin{subfigure}{0.4\textwidth}
        $$\begin{array}{|c|c|c|c|c|c|c|}\hline
            4,5 & 3,5 & 1,2 & 1,7 & 2,6 & 3 & 4 \\ \hline
            3 &  & 1,5 & 1,6 & 4,7 & 4 & 5 \\ \hline
            4,7 & 1,6 &  & 5 & 2 & 1 & 2 \\ \hline
            2,2 & 1,7 & 6 &  & 1 & 5 & 3 \\ \hline
            1,6 & 3,4 & 7 & 2 &  & 2 & 1 \\ \hline
            1 & 5 & 4 & 2 & 3 &  &  \\ \hline
            5 & 4 & 2 & 3 & 1 &  &  \\ \hline
        \end{array}$$
        \caption{$|A_\ell| = 2$}
    \end{subfigure}
    \begin{subfigure}{0.4\textwidth}
        $$\begin{array}{|c|c|c|c|c|c|c|c|}\hline
            3,5 & 4,8 & 1,6 & 5,7 & 2,2 & 3 & 4 & 1 \\ \hline
            3,4 &  & 1,8 & 3,6 & 1,7 & 4 & 5 & 5 \\ \hline
            1,8 & 1,7 &  & 2 & 6 & 5 & 2 & 4 \\ \hline
            2,6 & 1,5 & 7 &  & 8 & 1 & 3 & 2 \\ \hline
            1,7 & 4,6 & 2 & 8 &  & 2 & 1 & 3 \\ \hline
            2 & 5 & 4 & 1 & 3 &  &  &  \\ \hline
            5 & 3 & 2 & 1 & 4 &  &  &  \\ \hline
            4 & 3 & 5 & 2 & 1 &  &  &  \\ \hline
        \end{array}$$
        \caption{$|A_\ell| = 3$}
    \end{subfigure}

    \begin{subfigure}{0.5\textwidth}
        $$\begin{array}{|c|c|c|c|c|c|c|c|c|}\hline
            2,4 & 3,6 & 4,9 & 1,8 & 2,7 & 5 & 3 & 1 & 5 \\ \hline
            3,5 &  & 1,8 & 6,7 & 1,9 & 4 & 5 & 4 & 3 \\ \hline
            2,7 & 5,8 &  & 9 & 6 & 1 & 1 & 2 & 4 \\ \hline
            1,6 & 1,9 & 7 &  & 8 & 3 & 2 & 5 & 2 \\ \hline
            8,9 & 1,7 & 6 & 2 &  & 2 & 4 & 3 & 1 \\ \hline
            5 & 4 & 1 & 2 & 3 &  &  &  &  \\ \hline
            4 & 3 & 2 & 5 & 1 &  &  &  &  \\ \hline
            3 & 4 & 5 & 1 & 2 &  &  &  &  \\ \hline
            1 & 5 & 2 & 3 & 4 &  &  &  &  \\ \hline
        \end{array}$$
        \caption{$|A_\ell| = 4$}
    \end{subfigure}
        \caption{Outline arrays for $F_\ell$ where $\ell\in[a-b]$}
        \label{fig: ILS 5 row bits}
    \end{figure}

    Observe that $F = \sum_{\ell=1}^{b} F_\ell$. Thus, by \Cref{lemma: sum freq arrays}, an outline array exists corresponding to $F$.
\end{proof}

By putting the previous results together, we construct an $\ILS(n;h_2\dots h_k)$ for all partitions $(h_2\dots h_k)$ where $n=h_2+\sum_{i=2}^kh_i$.

\begin{theorem}
\label{thm: ILS con h3<=r/4}
    Let $(h_1\dots h_k)$ be an integer partition, where $h_1=h_2\geq h_3 \ge \dots \ge h_k$, $r=\sum_{i=3}^k h_i$ and $h_3 \le \frac{1}{4}(r+1)$. There exists an $\ILS(2h_1+r;h_2\dots h_k)$.
\end{theorem}
\begin{proof}
    If $2h_1\leq r+1-2h_3$, then by \Cref{circulant_construction} or \Cref{even r construction} there is an $\LS((2h_1)h_3\dots h_k)$ which can be changed to an $\ILS(2h_1+r;h_2\dots h_k)$ by replacing the subsquare of order $2h_1$ with an inflation by $h_1$ of the order 2 latin square shown below.
    $$\begin{array}{|c|c|}\hline
        2 & 1 \\\hline
        1 & 2 \\\hline
    \end{array}$$

    If $2h_1 > r+1-2h_3$, then let
    $$g = \begin{cases}
        \frac{1}{2}(r+1-2h_3), & \text{if $r$ is odd,}\\
        \frac{1}{2}(r-2h_3), & \text{if $r$ is even,}
    \end{cases}$$
    and let $L$ be an $\ILS(2g+r;gh_3\dots h_k)$ as constructed above and take $O$ to be the reduction modulo $(P,P,P)$ of $L$ for $P = (g^2h_3\dots h_k)$. Since $O(1,1)$ contains only $g^2$ copies of 2 and $O(1,2)$ and $O(2,1)$ contain $g^2$ copies of 1, removing all entries from $(1,2)$, $(2,1)$ and $(i,i)$ for all $i\in[k]$ makes $O$ an outline array for the frequency array $F$ of order $k$ where for all $i,j\in[k]$
    $$F(i,j) = \begin{cases}
        h_ih_j, & \text{if $i,j\geq 3$ and $i\neq j$,}\\
        gh_j, & \text{if $i\in[2]$ and $j\geq 3$,}\\
        gh_i, & \text{if $j\in[2]$ and $i\geq 3$,}\\
        0, & \text{otherwise.}
    \end{cases}$$
    Let $F^*$ be a frequency array of order $k$ where for all $i,j\in[k]$
    $$F^*(i,j) = \begin{cases}
        h_1^2, & \text{if $i,j\in[2]$ and $(i,j)\neq (2,2)$,}\\
        (h_1-g)h_j, & \text{if $i\in[2]$ and $j\geq 3$,}\\
        (h_1-g)h_i, & \text{if $j\in[2]$ and $i\geq 3$,}\\
        0, & \text{otherwise.}
    \end{cases}$$
    By \Cref{lemma: ILS L thing}, there exists an outline array $O^*$ corresponding to $F^*$ if $h_1^2\geq (h_1-g)\sum_{i=3}^kh_i$.

    Observe that for odd $r$
    \begin{align*}
        h_1^2 - (h_1-g)r &= h_1^2 - (h_1-\frac{1}{2}r - \frac{1}{2} + h_3)r\\
        &= (h_1-\frac{1}{2}r)^2 + \frac{1}{4}r(r+2-4h_3)\\
        &\geq 0
    \end{align*}
    since $h_3\leq \frac{1}{4}(r+1)$, and for even $r$
    \begin{align*}
        h_1^2 - (h_1-g)r &= h_1^2 - (h_1-\frac{1}{2}r + h_3)r\\
        &= (h_1-\frac{1}{2}r)^2 + \frac{1}{4}r(r-4h_3)\\
        &\geq 0.
    \end{align*}

    Thus, the outline array $O^*$ exists, and taking the cell-wise union of $O$ with $O^*$ gives an outline array for an $\ILS(2h_1+r;h_2\dots h_k)$.
\end{proof}

\begin{theorem}
\label{thm: ILS con h_1}
    Let $(h_1\dots h_k)$ be an integer partition, where $h_1=h_2\geq h_3 \ge \dots \ge h_k$ and $r=\sum_{i=3}^k h_i$. There exists an $\ILS(2h_1+r;h_2\dots h_k)$.
\end{theorem}
\begin{proof}
    If $h_3\leq \frac{1}{4}(r+1)$, then an $\ILS(2h_1+r;h_2\dots h_k)$ exists by \Cref{thm: ILS con h3<=r/4}. Therefore we can assume that $h_3 >\frac{1}{4}(r+1)$.
    
    If $h_3=h_4=h_5=h_6$, then $h_3\leq \frac{1}{4}r$. Thus, we assume that $h_3>h_6$.

    \medskip

    We first consider the case where $h_3=h_4=h_5$; let $m = r-3h_3$.
    
    Since $h_3>\frac{1}{4}(r+1)$, $h_3>m+1$. Set $g_3 = g_4 = g_5 = m+1$, $c = h_3-g_3$, $h_1' = \min\{h_1,3h_3+2g_3\}$ and $g_1=g_2 = h_1'-c$. Then $g_3 = \frac{1}{4}(3g_3 + (m+1)) = \frac{1}{4}(3g_3+1+\sum_{i=6}^kh_i)$, and so an $\ILS(2g_1+3g_3+m;g_1g_3^3h_6\dots h_k)$ exists by \Cref{thm: ILS con h3<=r/4}. Let $O$ be the reduction modulo $(Q,Q,Q)$ of this latin square where $Q = (g_1^2g_3^3h_6\dots h_k)$ and $O(i,i)$ contains only symbol $i$ for all $i\in[k]\setminus[1]$.
    
    Set $a = \min\{2(h_3^2-g_3^2),h_1'h_3-g_1g_3\}$. Let $F_1$ be a frequency array of order $k$ where for all $i,j\in[k]$
    $$F_1(i,j) = \begin{cases}
        a, & \text{if $i\in[5]$ and $j\in[2]$, or $j\in[5]$ and $i\in[2]$, and $(i,j)\neq(2,2)$,}\\
        h_3^2 - g_3^2, & \text{if $i,j\in2+[3]$ and $i\neq j$,}\\
        ch_j, & \text{if $i\in[5]$ and $j>5$,}\\
        ch_i, & \text{if $j\in[5]$ and $i>5$,}\\
        0, & \text{otherwise.}
    \end{cases}$$
    Observe that $h_3^2 - g_3^2 = (h_3 + g_3)c >(m+1)c$, thus $4(h_3^2 - g_3^2) \geq cm$. Since $a\leq 2(h_3^2-g_3^2)$, an outline array $O_1$ corresponding to $F_1$ exists by \Cref{lemma: ILS L thing}.

    Let $F_2$ and $F_3$ be frequency arrays of order $k$ where for all $i,j\in[k]$
    $$F_2(i,j) = \begin{cases}
        (h'_1)^2-g_1^2-a, & \text{if $i,j\in[2]$ and $i,j\neq(2,2)$,}\\
        h'_1h_3-g_1g_3-a, & \text{if $i\in2+[3]$ and $j\in[2]$, or $j\in2+[3]$ and $i\in[2]$,}\\
        0, & \text{otherwise,}
    \end{cases}$$
    and
    $$F_3(i,j) = \begin{cases}
        h_1^2-(h'_1)^2, & \text{if $i,j\in[2]$ and $i,j\neq(2,2)$,}\\
        (h_1-h_1')h_j, & \text{if $i\in[2]$ and $j>2$,}\\
        (h_1-h_1')h_i, & \text{if $j\in[2]$ and $i>2$,}\\
        0, & \text{otherwise.}
    \end{cases}$$
    If $a = h_1h_3-g_1g_3$, then $F_2(i,j) = 0$ unless $i,j\in[2]$. Otherwise, $a = 2(h_3^2 - g_3^2)$ and $((h'_1)^2-g_1^2-a) - 3(h'_1h_3-g_1g_3-a) = c(3h_3+2g_3-h'_1)\geq 0$ since $h'_1\leq 3h_3+2g_3$. In either case, an outline array $O_2$ corresponding to $F_2$ always exists by \Cref{lemma: ILS L thing}.

    If $h'_1 = h_1$ then all entries of $F_3$ are 0 and the corresponding outline array $O_3$ is an empty array. Otherwise, $h'_1 = 3h_3 + 2g_3 < h_1$ and $(h_1-h'_1)(h_1 + h'_1) = (h_1 - h'_1)(h_1 + 3h_3 + 2g_3) > (h_1 - h'_1)(3h_3 + m)$. Thus, the outline array $O_3$ exists by \Cref{lemma: ILS L thing}.

    Taking the cell-wise union of $O$ with $O_1$, $O_2$ and $O_3$, and increasing the number of copies of $i$ in cell $(i,i)$ for $i\in[5]\setminus[1]$, gives an outline square for an $\ILS(2h_1+r;h_2\dots h_k)$. Therefore, an $\ILS(2h_1+r;h_2\dots h_k)$ exists for all partitions $(h_1\dots h_k)$ where $h_1=h_2$ and $h_3=h_4=h_5$.
    
    \medskip

    We next consider the case where $h_3 = h_4 > h_5$; let $m = r-2h_3$.
    
    Since $h_3>\frac{1}{4}(r+1)$, $h_3>\frac{1}{2}(m+1)$. Let $c = h_3-h_5$ and $g_1 = g_2 = h_1-c$. There exists an $\ILS(2g_1+3h_5+m;g_2h_5^3h_6\dots h_k)$ either from above or from \Cref{thm: single subsquare} if $h_5=0$. Let $O$ be the reduction modulo $(Q,Q,Q)$ of this latin square where $Q = (g_1^2h_5^3h_6\dots h_k)$ and $O(i,i)$ contains only symbol $i$ for all $i\in[k]\setminus[1]$.
    
    Let $F_1$ and $F_2$ be frequency arrays of order $k$ where for all $i,j\in[k]$
    $$F_1(i,j) = \begin{cases}
        h_3^2 - h_5^2, & \text{if $i,j\in[4]$ and $i\neq j$ unless $i=j=1$,}\\
        ch_j, & \text{if $i\in[4]$ and $j\geq 5$,}\\
        ch_i, & \text{if $j\in[4]$ and $i\geq 5$,}\\
        0, & \text{otherwise,}
    \end{cases}$$
    and
    $$F_2(i,j) = \begin{cases}
        h_1^2 - h_3^2, & \text{if $i,j\in[2]$ and $i,j\neq(2,2)$,}\\
        h_1h_3 - h_3^2, & \text{if $i\in2+[2]$ and $j\in[2]$, or $j\in2+[2]$ and $i\in[2]$,}\\
        0, & \text{otherwise.}
    \end{cases}$$
    Since $3(h_3^2 - h_5^2) = 3(h_3+h_5)c\geq cm$ and $(h_1^2 - h_3^2) - 2(h_1h_3 - h_3^2) = (h_1-h_3)^2\geq 0$, outline arrays $O_1$ and $O_2$ corresponding to $F_1$ and $F_2$ exist by \Cref{lemma: ILS L thing}.

    Taking the cell-wise union of $O$ with $O_1$ and $O_2$, and increasing the number of copies of $i$ in cell $(i,i)$ for $i\in[4]\setminus[1]$, gives an outline square for an $\ILS(2h_1+r;h_2\dots h_k)$. Thus, an $\ILS(2h_1+r;h_2\dots h_k)$ exists for all partitions $(h_1\dots h_k)$ where $h_1=h_2$ and $h_3=h_4$.

    \medskip
    
    Finally, we consider the case where $h_3>h_4$; let $m = \sum_{i=4}^kh_i$.
    
    It follows that $r = m+h_3$ and $h_3>\frac{1}{3}(m+1)$. Let $g_3 = \max\{h_4,\lfloor\frac{1}{3}(m+1)\rfloor\}$, $h_3 = g_3+c$ and $g_1 = g_2 = h_1-c \geq \lfloor\frac{1}{3}(m+1)\rfloor$. Then an $\ILS(2g_1+g_3+m;g_2g_3h_4\dots h_k)$ exists from the previous case or from \Cref{thm: ILS con h3<=r/4}. Let $O$ be the reduction modulo $(Q,Q,Q)$ of this latin square where $Q = (g_1^2g_3h_4\dots h_k)$ and $O(i,i)$ contains only symbol $i$ for all $i\in[k]\setminus[1]$.
    
    Let $F_1$ and $F_2$ be frequency arrays of order $k$ where for all $i,j\in[k]$
    $$F_1(i,j) = \begin{cases}
        h_1h_3 - g_1g_3, & \text{if $i,j\in[3]$ and $i\neq j$ unless $i=j=1$,}\\
        ch_j, & \text{if $i\in[3]$ and $j\geq 4$,}\\
        ch_i, & \text{if $j\in[3]$ and $i\geq 4$,}\\
        0, & \text{otherwise,}
    \end{cases}$$
    and
    $$F_2(i,j) = \begin{cases}
        h_1^2 - h_1h_3 + g_1g_3, & \text{if $i,j\in[2]$ and $i,j\neq(2,2)$,}\\
        0, & \text{otherwise.}
    \end{cases}$$
    Note that $2(h_1h_3 - g_1g_3) = 2(h_1+h_3-c)c = 2(g_1+h_3)c \geq cm$ since $g_1\geq \lfloor\frac{1}{3}(m+1)\rfloor$ and $h_3>\frac{1}{3}(m+1)$, and $h_1^2 - g_1^2 = (2h_1-c)c \geq (h_1+h_3-c)c = h_1h_3 - g_1g_3$. Thus, outline arrays $O_1$ and $O_2$ corresponding to $F_1$ and $F_2$ exist by \Cref{lemma: ILS L thing}.

    Taking the cell-wise union of $O$ with $O_1$ and $O_2$, and increasing the number of copies of 2 and 3 in the subsquares in cells $(2,2)$ and $(3,3)$ respectively gives an outline square for an $\ILS(2h_1+r;h_2\dots h_k)$.
\end{proof}

Returning to the question at the start of this section, we prove here that for any partition $(h_2\dots h_k)$ it is sufficient to take $h_1\geq h_2$.

\begin{theorem}
\label{thm: ILS con}
    For any partition $(h_1\dots h_k)$ of $n$ where $h_1\geq h_2\geq \dots\geq h_k$, an $\ILS(n;h_2\dots h_k)$ exists.
\end{theorem}
\begin{proof}
    Since $h_1\geq h_2$ there exists an $\ILS(n;h_2\dots h_k1^{h_1-h_2})$ by \Cref{thm: ILS con h_1}, which is also an $\ILS(n;h_2\dots h_k)$.
\end{proof}

\section*{Acknowledgement}
Funding: The author would like to acknowledge the support of the Australian Government through a Research Training Program (RTP) Scholarship.

\printbibliography

\end{document}